\documentclass[11pt]{amsart}

\usepackage{amsmath}

\usepackage{amssymb}

\usepackage{amsfonts}

\usepackage{amsthm}

\usepackage{enumerate}

 \usepackage{esint}
 
 \usepackage{tikz}

\usepackage{graphicx}

\usepackage{verbatim}

\newtheorem{innercustomthm}{Theorem}
\newenvironment{customthm}[1]
  {\renewcommand\theinnercustomthm{#1}\innercustomthm}
  {\endinnercustomthm}

\newtheorem{innercustomprop}{Proposition}
\newenvironment{customprop}[1]
  {\renewcommand\theinnercustomprop{#1}\innercustomprop}
  {\endinnercustomprop}
  
  \newtheorem{innercustomlemma}{Lemma}
\newenvironment{customlemma}[1]
  {\renewcommand\theinnercustomlemma{#1}\innercustomlemma}
  {\endinnercustomlemma}

\newcommand{\Norm}[1]{ \left\|  #1 \right\| }

\newcommand{\floor}[1]{\lfloor #1 \rfloor }

\newcommand{\Be}{\begin{equation}}

\newcommand{\Ee}{\end{equation}}

\newcommand{\Bm}{\begin{multline}}

\newcommand{\Em}{\end{multline}}

\newcommand{\Bea}{\begin{eqnarray}}

\newcommand{\Eea}{\end{eqnarray}}

\newcommand{\Beas}{\begin{eqnarray*}}

\newcommand{\Eeas}{\end{eqnarray*}}

\newcommand{\Benu}{\begin{enumerate}}

\newcommand{\Eenu}{\end{enumerate}}

\newcommand{\Bi}{\begin{itemize}}

\newcommand{\Ei}{\end{itemize}}

\def\intslash{\fint}

\def\intslash{\rlap{\kern  .32em $\mspace {.5mu}\backslash$ }\int}

\def\qsl{{\rlap{\kern  .32em $\mspace {.5mu}\backslash$ }\int_{Q_x}}}

\def\Norm#1{{ \left\|  #1 \right\| }}

\def\floor#1{{\lfloor #1 \rfloor }}

\def\emph#1{{\it #1 }}

 at 10 true pt

\def\be#1{\begin{equation}\label{ #1}}

\def\endeq{\end{equation}}

\def\endal{\end{align}}

\def\bas{\begin{align*}}

\def\eas{\end{align*}}

\def\bi{\begin{itemize}}

\def\ei{\end{itemize}}

\def\emph#1{{\it #1}}

\def\textbf#1{{\bf #1}}

\theoremstyle{plain}

   \newtheorem{theorem}{Theorem}[section]

   \newtheorem{conjecture}[theorem]{Conjecture}

   \newtheorem{proposition}[theorem]{Proposition}

   \newtheorem{lemma}[theorem]{Lemma}

   \newtheorem{theorem*}{Theorem}

\theoremstyle{remark}

   \newtheorem{remark}[theorem]{Remark}

\theoremstyle{definition}

   \newtheorem{definition}[theorem]{Definition}

\numberwithin{equation}{section}

\begin{document}
\title[Bochner-Riesz multipliers associated with convex domains]{New $L^p$ bounds for Bochner-Riesz multipliers associated with convex planar domains with rough boundary}

\author[L. Cladek]{Laura Cladek}

\address{L. Cladek, Department of Mathematics\\ University of Wisconsin-Madison\\480 Lincoln Drive, Madison, WI 53706, USA}

\email{cladek@math.wisc.edu}

\subjclass[2010]{42B15}

\begin{abstract}
We consider generalized Bochner-Riesz multipliers of the form $(1-\rho(\xi))_+^{\lambda}$ where $\rho(\xi)$ is the Minkowski functional of a convex domain in $\mathbb{R}^2$, with emphasis on domains for which the usual Carleson-Sj\"{o}lin $L^p$ bounds can be improved. We produce convex domains for which previous results due to Seeger and Ziesler are not sharp. We identify two key properties of convex domains that lead to improved $L^p$ bounds for the associated Bochner-Riesz operators. First, we introduce the notion of the ``additive energy" of the boundary of a convex domain. Second, we associate a set of directions to a convex domain and define a sequence of Nikodym-type maximal operators corresponding to this set of directions. We show that domains that have low higher order additive energy, as well as those which have asymptotically good $L^q$ bounds for the corresponding sequence of Nikodym-type maximal operators where $q=(p^{\prime}/2)^{\prime}$, have improved $L^p$ bounds for the associated Bochner-Riesz operators over those proved by Seeger and Ziesler.
\end{abstract}

\thanks{The author would like to thank Andreas Seeger for introducing this problem, and for his guidance and many helpful discussions. Research supported in part by NSF Research and Training grant DMS 1147523}
\maketitle

\section{Introduction}
The Bochner-Riesz operators $R_{\lambda}$ are defined via the Fourier transform by
\begin{align*}
&\mathcal{F}[R_{\lambda}f](\xi)=(1-|\xi|)^{\lambda}_+\widehat{f}(\xi),\qquad \lambda>0,
\\
&\mathcal{F}[R_{0}f](\xi)=\chi_{B_0(1)}(\xi)\widehat{f}(\xi),
\end{align*}
where $\chi_{B_0(1)}$ denotes the characteristic function of the ball of radius $1$ centered at the origin.
In two dimensions the $L^p$ mapping properties of $R_{\lambda}$ are completely known. As first shown by Fefferman in \cite{feff} and later by C\'{o}rdoba in \cite{cor}, if $\lambda>0$ then $R_{\lambda}$ is bounded on $L^p(\mathbb{R}^2)$ if and only if $\lambda>\max((|\frac{2}{p}-1|-\frac{1}{2}), 0)$. It was also shown by Fefferman in \cite{feff2} that $R_0$ is bounded on $L^p(\mathbb{R}^2)$ if and only if $p=2$. One may also consider the following generalization of the two-dimensional Bochner-Riesz operators. Let $\Omega\subset\mathbb{R}^2$ be a bounded, open convex set containing the origin, and let $\rho$ be its Minkowski functional, defined as
\begin{align*}
\rho(\xi)=\inf\{t>0:\,t^{-1}\xi\in\Omega\}.
\end{align*}
Define the generalized Bochner-Riesz operators $T_{\lambda}$ associated to $\Omega$ by
\begin{align*}
&\mathcal{F}[T_{\lambda}f](\xi)=(1-\rho(\xi))^{\lambda}_+\widehat{f}(\xi), \qquad \lambda>0,
\\&\mathcal{F}[T_{0}f](\xi)=\chi_{\Omega}(\xi)\widehat{f}(\xi),
\end{align*}
where $\chi_{\Omega}$ denotes the characteristic function of $\Omega$. Note that in the special case that $\Omega$ is the unit disk, $T_{\lambda}$ is simply $R_{\lambda}$. We emphasize that no further regularity of $\partial\Omega$ is assumed, and for general convex domains $\Omega$ the boundary $\partial\Omega$ need only be Lipschitz.
\newline
\indent
For domains with smooth boundary, the $L^p$ mapping properties of $T_{\lambda}$ were shown by Sj\"{o}lin in \cite{sjolin} to be identical to those of  $R_{\lambda}$. However, for certain convex domains with rough boundary the $L^p$ mapping properties of $T_{\lambda}$ may be improved. In \cite{pod2}, Podkorytov showed that in the case that $\Omega$ is a polyhedron in $\mathbb{R}^d$, $T_{\lambda}$ is bounded on $L^p$ for $1\le p\le\infty$ and for all $\lambda> 0$. In \cite{sz}, Seeger and Ziesler proved a sufficient criterion for $L^p$ boundedness of Bochner-Riesz multipliers associated to general convex domains in $\mathbb{R}^2$. Their results depended on a parameter similar to the upper Minkowski dimension of $\partial\Omega$, defined by a family of ``balls", or caps, and we give a definition below. This parameter may be thought of as measuring how ``curved" the boundary of $\Omega$ is.
\newline
\indent
For any $p\in\partial\Omega$, we say that a line $\ell$ is a \textit{supporting line for $\Omega$ at $p$} if $\ell$ contains $p$ and $\Omega$ is contained in the half plane containing the origin with boundary $\ell$. Let $\mathcal{T}(\Omega, p)$ denote the set of supporting lines for $\Omega$ at $p$. Note that if $\partial\Omega$ is $C^1$, then $\mathcal{T}(\Omega, p)$ has exactly one element, the tangent line to $\partial\Omega$ at $p$. For any $p\in\partial\Omega$, $\ell\in \mathcal{T}(\Omega, p)$, and $\delta>0$, define
\begin{align}
B(p, \ell, \delta)=\{x\in\partial\Omega: \text{dist}(x, \ell)<\delta\}.
\end{align}
Let
\begin{align}
\mathcal{B}_{\delta}=\{B(p, \ell, \delta):\,p\in\partial\Omega, \ell\in\mathcal{T}(\Omega, p)\},
\end{align}
and let $N(\Omega, \delta)$ be the minimum number of balls $B\in\mathcal{B}_{\delta}$ needed to cover $\partial\Omega$. Let
\begin{align}
\kappa_{\Omega}=\limsup_{\delta\to 0}\frac{\log N(\Omega, \delta)}{\log\delta^{-1}}.
\end{align}
It is easy to show using Cauchy-Schwarz that for any convex domain $\Omega$, $0\le\kappa_{\Omega}\le\frac{1}{2}$. If $\partial\Omega$ is smooth, then $\kappa_{\Omega}=1/2$. This can be seen by noting that there is a point where $\partial\Omega$ has nonvanishing curvature, and near this point the contribution to $N(\Omega, \delta)$ is $\approx\delta^{-1/2}$. \newline
\indent
We now state the main result from \cite{sz}, due to Seeger and Ziesler.
\begin{customthm}{A}[\cite{sz}]\label{sz1}
Suppose that $1\le p\le \infty$, $\lambda>0$ and $\lambda>\kappa_{\Omega}(4|1/p-1/2|-1)$. Then $T_{\lambda}$ is bounded on $L^p(\mathbb{R}^2)$. 
\end{customthm}
Note that as $\kappa_{\Omega}$ gets smaller, the range of $p$ for which $T_{\lambda}$ is bounded improves, so for rough domains it is possible to do much better than the optimal result for domains with smooth boundary. The authors of \cite{sz} also showed that for each $\kappa\in(0, 1/2)$ there is a convex domain $\Omega$ with $\kappa_{\Omega}=\kappa$ for which Theorem \ref{sz1} is sharp.

\begin{customthm}{B}[\cite{sz}]\label{sz2}
Let $0<\kappa<1/2$. Then there exists a convex domain $\Omega$ with $C^{1, \frac{\kappa}{(1-\kappa)}}$ boundary satisfying $\kappa_{\Omega}=\kappa$ so that for $1\le p<4/3$ the operator $T_{\lambda}$ associated to $\Omega$ is bounded on $L^p(\mathbb{R}^2)$ if and only if $\lambda>\kappa_{\Omega}(4/p-3)$.
\end{customthm}
We will show that for every $\kappa\in (0, 1/2)$ sufficiently small there exists a convex domain $\Omega$ with $\kappa_{\Omega}=\kappa$ for which Theorem \ref{sz1} is not sharp. 
\begin{theorem}\label{2ndmainprop}
Let $m\ge 2$ be an integer. Let $\kappa\in (0, \frac{1}{4m-2}]$. Then there exists a convex domain $\Omega$ with $\kappa_{\Omega}=\kappa$ so that for $1\le p\le\frac{2m}{2m-1}$, $T_{\lambda}$ is bounded on $L^p(\mathbb{R}^2)$ if $\lambda>\kappa_{\Omega}(\frac{m+2}{p}-m-1)$, and for $4/3\le p\le 4$, $T_{\lambda}$ is bounded on $L^p(\mathbb{R}^2)$ if and only if $\lambda>0$.
\end{theorem}
Note that the case $m=2$ above corresponds to Theorem \ref{sz1}, and that if $m\ge 3$ Theorem \ref{2ndmainprop} gives an improvement over Theorem \ref{sz1} in the range $1\le p<\frac{2m}{2m-1}$ (and of course, in the dual range as well). Theorem \ref{2ndmainprop} demonstrates that how ``curved" the boundary of a convex planar domain is, as measured by the parameter $\kappa_{\Omega}$, does not alone determine the $L^p$ mapping properties of the associated Bochner-Riesz operators, but rather there must be other properties of $\Omega$ that play a role. Theorem \ref{2ndmainprop} also shows that there exist domains with $\kappa_{\Omega}>0$ such that $p_{\text{crit}}<4/3$.
\newline
\indent
In the proof of Theorem \ref{sz2}, a crucial property of the domains constructed was that their boundaries contained long arithmetic progressions at every scale, in the sense that for every $\delta>0$ the boundary could be covered by essentially disjoint balls in $\mathcal{B}_{\delta}$ such that a large sequence of consecutive balls were essentially equally spaced in a single coordinate direction. We now describe a simplified version of their construction, removing the requirement that $\Omega$ has $C^{1, \frac{\kappa_{\Omega}}{1-\kappa_{\Omega}}}$ boundary in the statement of Theorem \ref{sz2}, as well as sharpness at the endpoint $\lambda=\kappa_{\Omega}(\frac{4}{p}-3)$. Choose a sequence of consecutive intervals $I_1, I_2, \ldots$ in $[0, 1]$ such that $I_k$ has length $2^{-k(1/2-\kappa_{\Omega})}$. For each $k$, let $E_k$ be a set of $2^{k\kappa_{\Omega}}$ essentially equally spaced points in $I_k$ at a distance $\approx 2^{-k/2}$ apart. Now for each $k$, let $\Omega_k$ denote the convex polygon with vertices 
\begin{align*}
\{(-1, 1); (-1, -2); (0, 1)\}\cup\{(x_1, x_1^2-2):\,x_1\in\bigcup_{1\le j\le k}E_j\}.
\end{align*}
Let $\Omega$ be the uniform limit of $\{\Omega_k\}$ as $k\to\infty$. Then one may show using similar arguments to those presented in \cite{sz} in the proof of Theorem \ref{sz2} that  whenever $1\le p<4/3$, $T_{\lambda}$ is bounded on $L^p(\mathbb{R}^2)$ only if $\lambda\ge\kappa_{\Omega}(\frac{4}{p}-3)$.
\newline
\indent
The domains we construct to prove Theorem \ref{2ndmainprop} will differ from those constructed in \cite{sz} to prove Theorem \ref{sz2} in in that they will exhibit ``low $n$-additive energy" at every scale for some $n>2$. To produce such domains will require a particular kind of ``fast-branching" Cantor-type construction. We define the $n$-additive energy of $\partial\Omega$ as follows.
\begin{definition}
Let $n\ge 2$ be an integer, and let $\Omega$ be a bounded, convex domain in $\mathbb{R}^2$. Let $\mathfrak{B}_{\delta}=\{B_1, B_2, \ldots, B_{N(\Omega, \delta)}\}$ be a collection of balls in $\mathcal{B}_{\delta}$ covering $\partial\Omega$. Let $\Xi_{\mathfrak{B}_{\delta}, n}$ be the smallest integer such that $\Xi_{\mathfrak{B}_{\delta}, n}=M_0^{2n}\cdot M_1$ and we may write $\mathfrak{B}_{\delta}$ as a union of $M_0$ subcollections $\mathfrak{B}_{\delta, 1}\ldots, \mathfrak{B}_{\delta, M_0}$ such that for each $1\le k\le M_0$, no point of $\mathbb{R}^2$ is contained in more than $M_1$ of the sets $B_{i_1}+\cdots+B_{i_n}$ where $B_{i_j}\in\mathfrak{B}_{\delta, k}$ for all $j$. Let $\Xi_{\delta, n}=\min_{\mathfrak{B}_{\delta}}(\Xi_{\mathfrak{B}_{\delta}, n})$, where the minimum is taken over all collections of balls in $\mathcal{B}_{\delta}$ covering $\partial\Omega$ with $\text{card}(\mathcal{B}_{\delta})=N(\Omega, \delta)$. We define the \textit{$n$-additive energy} of $\partial\Omega$ to be
\begin{align*}
\mathcal{E}_{n}(\partial\Omega)=\limsup_{\delta\to 0}\frac{\log(\Xi_{n, \delta})}{\log(\delta^{-1})}.
\end{align*}
\end{definition}
As a consequence of a lemma proven in \cite{sz}, we have $\mathcal{E}_2(\partial\Omega)=0$ for all convex domains $\Omega$. However, general convex domains fail to satisfy $\mathcal{E}_n(\partial\Omega)=0$ for some $n>2$, but the domains we construct will have this property.
\newline
\indent
To discuss a second important property that leads to improved $L^p$ bounds for generalized Bochner-Riesz multipliers, we first need to associate a set of directions to $\Omega$. Given $x\in\partial\Omega$, let $\theta_{x}, \theta_{x}^{\prime}$ be the slopes of two supporting lines at $x$ with maximum difference in angle (note there is a unique choice of two lines). We will allow slopes to be infinite to include the possibility of vertical lines. Note that if we choose $x$ so that $\partial\Omega$ may be parametrized near $x$ by $(\alpha, \gamma(\alpha))$, then $\theta_{x}$ and $\theta_{x}^{\prime}$ are simply the left and right derivatives of $\gamma$ evaluated at $x$. Let 
\begin{align*}
\Theta=\Theta(\Omega)=\{\theta_{x}, \theta_{x}^{\prime}:\,x\in\partial\Omega\}\subset\mathbb{R}\cup\{\infty\}.
\end{align*}
Define a sequence of Nikodym-type maximal operators $\{M_{\Theta, \delta}\}$ by
\begin{align*}
M_{\Theta, \delta}f(x)=\sup_{x\in R\in\mathcal{R}_{\delta}}\frac{1}{|R|}\int_Rf(y)\,dy,
\end{align*}
where $\mathcal{R}_{\delta}$ denotes the set of all rectangles of eccentricity $\le\delta^{-1}$ with long side having slope in $\Theta$. We will be interested in how $\Norm{M_{\Theta, \delta}}_{L^p\to L^p}$ behaves as $\delta\to 0$. It was shown by Bateman in \cite{bateman} that if $M_{\Theta}$ denotes the directional maximal operator corresponding to $\Theta$, then $M_{\Theta}$ is unbounded on $L^p$ for all $p$ such that $1\le p<\infty$ unless $\Theta$ is a union of finitely many lacunary sets of finite order, and it is easy to show that any domain $\Omega$ with $\Theta(\Omega)$ a union of finitely many lacunary sets of finite order satisfies $\kappa_{\Omega}=0$. Thus for all domains with $\kappa_{\Omega}>0$ we must necessarily have that $\Norm{M_{\Theta, \delta}}_{L^p\to L^p}\to\infty$ as $\delta\to 0$. 

\begin{definition}
We say that $\Theta$ is \textit{$p$-sparse} if
\begin{align*}
\Norm{M_{\Theta, \delta}}_{L^p\to L^p}=O(\delta^{-\epsilon})
\end{align*}
for every $\epsilon>0$.
\end{definition}
It follows immediately by a theorem of C\'{o}rdoba (see \cite{cor2}) regarding the $L^2$ bounds for the Nikodym maximal function in $\mathbb{R}^2$ that every $\Theta$ is $p$-sparse for $2\le p<\infty$. We will see that if $\Theta(\Omega)$ is $p$-sparse for some $p<2$, then $T_{\lambda}$ satisfies improved $L^p$ bounds over those stated in Theorem \ref{sz1}. However, it is unclear whether the domains we construct are $p$-sparse for some $p<2$; hence construction of domains with $\kappa_{\Omega}>0$ that are $p$-sparse for some $p<2$ remains an interesting open question.
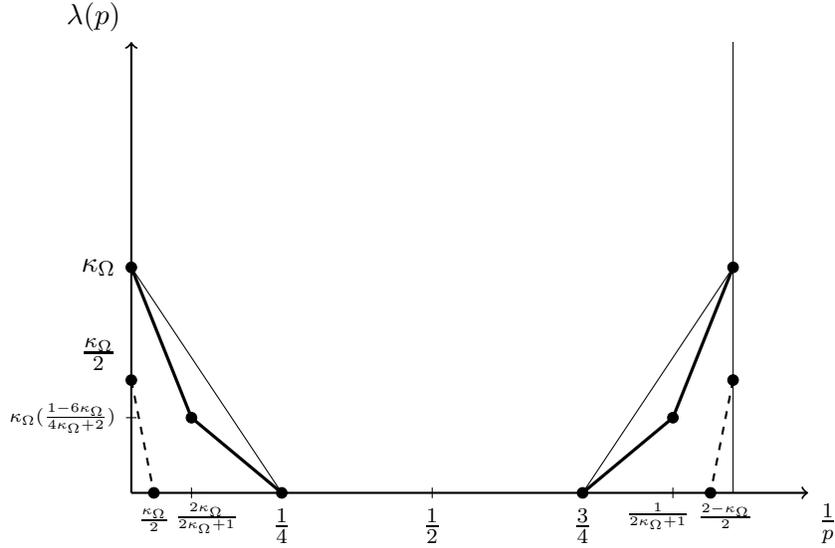
\begin{figure}
\begin{tikzpicture}
\draw[thick, ->] (0, 0)--(9, 0) node[anchor=north west] {$\frac{1}{p}$};
\draw[thick, ->] (0, 0)--(0, 6) node[anchor=south east] {$\lambda(p)$};

 \draw(4 cm, 2pt) -- (4 cm, -2pt) node[anchor=north] {$\frac{1}{2}$};
 \draw(2 cm, 2pt) -- (2 cm, -2pt) node[anchor=north] {$\frac{1}{4}$};
 \draw(6 cm, 2pt) -- (6 cm, -2pt) node[anchor=north] {$\frac{3}{4}$};

 \draw(1.0 cm, 0pt) -- (1.0 cm, 0pt) node[anchor=north] {\tiny $\frac{2\kappa_{\Omega}}{2\kappa_{\Omega}+1}$};
  \draw(7 cm, 0pt) -- (7 cm, 0pt) node[anchor=north] {\tiny $\frac{1}{2\kappa_{\Omega}+1}$};
 \filldraw (8, 3) circle (2pt);
 \filldraw (0, 3) circle (2pt);
 \filldraw (2, 0) circle (2pt);
 \filldraw (6, 0) circle (2pt);
 \draw (8, 3)--(6, 0);
 \draw (0, 3)--(2, 0); 

 \filldraw(0, 1.5) circle (2pt);

\draw (7.9 cm, 0pt)--(7.9 cm, 0pt) node[anchor=north] {\tiny$\frac{2-\kappa_{\Omega}}{2}$};
\draw (.3 cm, 2pt)--(.3 cm, -2pt) node[anchor=north] {\tiny$\frac{\kappa_{\Omega}}{2}$};
\filldraw (.3, 0) circle (2pt);
\filldraw (7.7, 0) circle (2pt);
\filldraw (8, 1.5) circle (2pt);
\draw(.8 cm, 2pt) -- (.8 cm, -2pt);
\draw(7.2 cm, 2pt) -- (7.2 cm, -2pt);
\draw(2pt, 3cm) -- (-2pt, 3cm) node[anchor=east] {$\kappa_{\Omega}$};
\draw(2pt, 1.5 cm) -- (-2pt, 1.5 cm) node[anchor=south east] {$\frac{\kappa_{\Omega}}{2}$};
\draw(2pt, 1 cm) -- (-2pt, 1 cm) node[anchor= east] {\tiny $\kappa_{\Omega}(\frac{1-6\kappa_{\Omega}}{4\kappa_{\Omega}+2})$};
\filldraw (.8, 1) circle (2pt);
\filldraw (7.2, 1) circle (2pt);
 \draw[very thick] (0, 3)--(.8, 1. 0);
  \draw[very thick] (.8, 1)--(2, 0);
 \draw[very thick] (7.2, 1)--(8, 3);
  \draw[very thick] (7.2, 1)--(6, 0);

\draw[thick, dashed] (8, 1.5)--(7.7, 0);
\draw[thick, dashed] (0, 1.5)--(.3, 0);
\draw (8, 0)--(8, 6);
  
\end{tikzpicture}
\caption{Here we sketch $\lambda(p)$ as a function of $\frac{1}{p}$ for certain convex domains, where $T_{\lambda}$ is bounded on $L^p$ for all $\lambda>\lambda(p)$. In this diagram, it is assumed that $\kappa_{\Omega}\le\frac{1}{10}$. The thin solid lines correspond to the domains constructed in \cite{sz} in the proof of Theorem \ref{sz2}; these domains exhibit long arithmetic progressions at every scale. The thick solid lines correspond to the domains that we construct to prove Theorem \ref{2ndmainprop} using a fast-branching Cantor-type construction; these lines as drawn are only valid if $\kappa_{\Omega}=\frac{1}{4m-2}$ for $m\ge 3$ an integer. The dashed lines represent lower bounds for general convex domains. That is, for any convex domain, $T_{\lambda}$ is unbounded on $L^p$ if $(\frac{1}{p}, \lambda)$ lies below the dashed lines.}
\label{fig1}
\end{figure}
We now formulate a general theorem on $L^p$ mapping properties of Bochner-Riesz means in terms of the $n$-additive energy of $\partial\Omega$ and the $L^q$-mapping properties of $M_{\Theta, \delta}$.

\begin{theorem}\label{genthm}
Let $\Omega$ be a convex domain in $\mathbb{R}^2$ containing the origin and let $\Theta$ be its associated set of directions. Let $n\ge 2$ be an integer. Suppose that $\mathcal{E}_n(\partial\Omega)=\alpha$ for some integer $0\le\alpha\le n\kappa_{\Omega}$ and that 
\begin{align*}
\Norm{M_{\Theta, \delta}}_{L^{\frac{n}{n-1}}(\mathbb{R}^2)\to L^{\frac{n}{n-1}}(\mathbb{R}^2)}\le C_{\epsilon}\delta^{-\beta-\epsilon}
\end{align*} 
for some $0\le\beta\le\kappa_{\Omega}(\frac{n-2}{n})$ and every $\epsilon>0$. Then for $1\le p\le \frac{2n}{2n-1}$, $T_{\lambda}$ is bounded on $L^p$ for $\lambda>\kappa_{\Omega}(\frac{2n}{p}-2n+1)+(\alpha/2n+\beta/2)(\frac{2np-2n}{p})$. 

\end{theorem}
Note that if $n=2$ we recover Theorem \ref{sz1}. One may check that if $n>2$, $\alpha=0$ and $\beta=\kappa_{\Omega}(\frac{n-2}{n})$ (i.e. $\beta$ is obtained by interpolating C\'{o}rdoba's estimate $\Norm{M_{\Theta, \delta}}_{L^2\to L^2}=O(\delta^{-\epsilon})$ with the trivial $L^1$ estimate $\Norm{M_{\Theta, \delta}}_{L^1\to L^1}=O(\delta^{-\kappa_{\Omega}})$), then Theorem \ref{genthm} gives improved bounds over those stated in Theorem \ref{sz1} in the range $1\le p\le \frac{2n}{2n-1}$. Fix a convex domain $\Omega$, and define 
\begin{align*}
p_{\text{crit}}:=\inf\{p: T_{\lambda}\text{ bounded on }L^p\text{ for all }\lambda>0\}.
\end{align*}
To achieve $p_{\text{crit}}<4/3$ using Theorem \ref{genthm} would require the construction of domains that simultaneously satisfy \textit{both} $\alpha=0$ and $\beta=0$ for some $n>2$. 
\newline
\indent
Finally, in Section \ref{khinsec} we will prove the following lower bounds for $T_{\lambda}$ for general convex domains.

\begin{theorem}\label{khin1}
Let $1\le p\le 2$. Let $\Omega\subset\mathbb{R}^2$ be a convex domain containing the origin, and let $T_{\lambda}$ denote the generalized Bochner-Riesz operator with exponent $\lambda$ associated to $\Omega$. Then $T_{\lambda}$ is unbounded on $L^p(\mathbb{R}^2)$ if $\lambda< 1-\frac{\kappa_{\Omega}}{2}-\frac{1}{p}$. In particular, $p_{\text{crit}}\ge\frac{2}{2-\kappa_{\Omega}}$. 
\end{theorem}
The proof will involve testing the operator on randomly defined functions, using Khinchine's inequality and Plancherel to estimate the $L^1$ and $L^2$ operator norms, respectively, and then interpolating.
\newline
\indent
We now give an overview of the layout of this paper. In Section \ref{prelim} we give useful preliminaries about convex domains in $\mathbb{R}^2$ and state some background results from \cite{sz}. In Section \ref{constrsec}, we construct the convex domains which we will later prove satisfy the statement of Theorem \ref{2ndmainprop}, and prove some results about the $n$-additive energy of their boundaries. In Section \ref{thmproof} we prove Theorem \ref{genthm}, which gives $L^p$ bounds for $T_{\lambda}$ as a consequence of certain conditions on the $n$-additive energy of $\partial\Omega$ and range of $q$ for which $\Theta(\Omega)$ is $q$-sparse. We also prove Theorem \ref{2ndmainprop} as a consequence of Theorem \ref{genthm}. In Section \ref{khinsec} we prove Theorem \ref{khin1}, which gives lower $L^p$  bounds on $T_{\lambda}$ for general convex domains with a given value of $\kappa_{\Omega}$. In Section \ref{concsec} we discuss some open questions which follow naturally from the results of this paper. 

\begin{remark}
All logarithms in this paper will be assumed to be base $2$, unless otherwise noted.
\end{remark}

\iffalse{\subsubsection*{Notation}
We now introduce some notation that will be used throughout this paper. If $I\subset\mathbb{R}$ is an interval, we let $I^{\ast}$ denote its double dilate, where the dilation is taken with respect to the center of $I$. All logarithms in this paper will be assumed to be base $2$ unless otherwise noted. }
\fi

\section{Preliminaries on convex domains in $\mathbb{R}^2$}\label{prelim}
In this section we give some useful background about convex domains in $\mathbb{R}^2$. All results in this section can be found in \cite{sz}, but we include them here for the sake of completeness. However, we will omit all proofs in this section, and the reader is encouraged to refer to \cite{sz} for proofs. 
\newline
\indent
Let $\Omega\subset\mathbb{R}^2$ be a bounded, open convex set containing the origin. Assume that $\Omega$ contains the ball of radius $4$ centered at the origin. Since $\Omega$ is bounded, there is an integer $M>0$ such that
\begin{align}\label{8est}
\{\xi:\,|\xi|\le 4\}\subset\Omega\subset\overline{\Omega}\subset\{\xi:\,|\xi|<2^M\}.
\end{align}
The following lemma is straightforward and can be proved using only elementary facts about convex functions.
\begin{customlemma}{C}[\cite{sz}]\label{elemlemma}
Suppose that $\Omega$ is a convex domain satisfying (\ref{8est}). Then $\partial\Omega\cap\{x:\,-1\le x_1\le 1,\,x_2<0\}$ can be parametrized by
\begin{align}
t\mapsto (t, \gamma(t)),\qquad{}-1\le t\le 1,
\end{align}
where
\begin{enumerate}
\item
\begin{align}
1<\gamma(t)<2^M,\qquad{} -1\le t\le 1.
\end{align}
\item $\gamma$ is a convex function on $[-1, 1]$, so that the left and right derivatives $\gamma_L^{\prime}$ and $\gamma_R^{\prime}$ exist everywhere in $(-1, 1)$ and
\begin{align}
-2^{M-1}\le\gamma_R^{\prime}(t)\le\gamma_L^{\prime}(t)\le 2^{M-1}
\end{align}
for $t\in [-1, 1]$. The functions $\gamma_L^{\prime}$ and $\gamma_R^{\prime}$ are decreasing functions; $\gamma_L^{\prime}$ and $\gamma_R^{\prime}$ are right continuous in $[-1, 1]$.
\item Let $\ell$ be a supporting line through $\xi\in\partial\Omega$ and let $n$ be an outward normal vector. Then
\begin{align}
|\left<\xi, n\right>|\ge 2^{-M}|\xi|.
\end{align}
\end{enumerate}
\end{customlemma}

\subsection*{Decomposition of $\partial\Omega$} As another preliminary ingredient, we need the decomposition of $\partial\Omega\cap \{x:\,-1\le x_1\le 1,\,x_2<0\}$ introduced in \cite{sz}. This decomposition allows us to write $\partial\Omega$ as a disjoint union of pieces on which $\partial\Omega$ is sufficiently ``flat", where the number of pieces in the decomposition is closely related to the covering numbers $N(\Omega, \delta)$. We inductively define a finite sequence of increasing numbers
\begin{align*}
\mathfrak{A}(\delta)=\{a_0, \ldots, a_Q\}
\end{align*}
as follows. Let $a_0=-1$, and suppose $a_0, \ldots, a_{j-1}$ are already defined. If
\begin{align}\label{case1}
(t-a_{j-1})(\gamma_L^{\prime}(t)-\gamma_R^{\prime}(a_{j-1}))\le\delta\text{ for all }t\in (a_{j-1}, 1])
\end{align}
and $a_{j-1}\le 1-2^{-M}\delta$, then let $a_j=1$. If (\ref{case1}) holds and $a_{j-1}>1-2^{-M}\delta$, then let $a_j=a_{j-1}+2^{-M}\delta$. If (\ref{case1}) does not hold, define
\begin{align*}
a_j=\inf\{t\in (a_{j-1}, 1]:\,(t-a_{j-1})(\gamma_L^{\prime}(t)-\gamma_R^{\prime}(a_{j-1}))>\delta\}.
\end{align*}
Now note that (\ref{case1}) must occur after a finite number of steps, since we have $|\gamma_L^{\prime}|, |\gamma_R^{\prime}|\le 2^{M-1}$, which implies that $|t-s||\gamma_L^{\prime}(t)-\gamma_R^{\prime}(s)|<\delta$ if $|t-s|<\delta 2^{-M}$. Therefore this process must end at some finite stage $j=Q$, and so it gives a sequence $a_0<a_1<\cdots<a_Q$ so that for $j=0, \ldots, Q-1$
\begin{align}\label{left} 
(a_{j+1}-a_j)(\gamma_L^{\prime}(a_{j+1})-\gamma_R^{\prime}(a_j))\le\delta,
\end{align}
and for $0\le j<Q-1$,
\begin{align}\label{right}
(t-a_j)(\gamma_L^{\prime}(t)-\gamma_R^{\prime}(a_j))>\delta\qquad\text{if }t>a_{j+1}.
\end{align}
For a given $\delta>0$, this gives a decomposition of 
\begin{align*}
\partial\Omega\cap\{x:\,-1\le x_1\le 1, x_2<0\}
\end{align*} 
into pieces
\begin{align*}
\bigsqcup_{n=0, 1, \ldots, Q-1}\{x\in\partial\Omega:\, x_1\in [a_n, a_{n+1}]\}.
\end{align*}
The number $Q$ in (\ref{left}) and (\ref{right}) is also denoted by $Q(\Omega, \delta)$. Let $R_{\theta}$ denote rotation by $\theta$ radians. The following lemma relates the numbers $Q(R_{\theta}\Omega, \delta)$ to the covering numbers $N(\Omega, \delta)$.
\begin{customlemma}{D}[\cite{sz}]\label{covlemma}
There exists a positive constant $C_M$ so that the following statements hold.
\begin{enumerate}
\item $Q(\Omega, \delta)\le C_M\delta^{-1/2}$.
\item $0\le\kappa_{\Omega}\le 1/2$.
\item For any $\theta$, 
\begin{align*}
Q(R_{\theta}\Omega, \delta)\le C_MN(\Omega, \delta)\log(2+\delta^{-1}).
\end{align*}
\item For $\nu=1, \ldots, 2^{2M}$ let $\theta_{\nu}=\frac{2\pi\nu}{2^{2M}}$. Then
\begin{align*}
C_M^{-1}N(\Omega, \delta)\le\sum_{\nu}Q(R_{\theta_{\nu}}\Omega, \delta)\le C_MN(\Omega, \delta)\log(2+\delta^{-1}).
\end{align*}
\end{enumerate}
\end{customlemma}

Finally, we state two results from \cite{sz} that we will need later in our proof of Theorem \ref{2ndmainprop}. The former is an $L^1$ estimate for the kernels of generalized Bochner-Riesz multipliers using a decomposition analogous to the standard decomposition of the (spherical) Bochner-Riesz multipliers into annuli. The latter is an $L^1$ kernel estimate corresponding to a finer decomposition of the generalized Bochner-Riesz multipliers associated with the decomposition of $\partial\Omega$ introduced above, as well as a pointwise majorization of a maximal function associated with this decomposition by a related Nikodym-type maximal function.
\begin{customprop}{E}[\cite{sz}]\label{propA}
Let $\Omega$ be a convex domain containing the origin. Let $\beta$ be a $C^2$ function supported on $(-1/2, 1/2)$ so that
\begin{align*}
|\beta^k(t)|\le 1, \qquad k=0, \ldots, 4.
\end{align*}
Let 
\begin{align*}
m_{\delta, \lambda}(\xi)=\delta^{\lambda}\beta\big(\frac{\delta^{-1}}{2}(1-\rho(\xi))\big).
\end{align*}
Then there is some $c>0$ such that for every $\delta>0$ sufficiently small,
\begin{align*}
\Norm{\mathcal{F}^{-1}[m_{\delta, \lambda}\widehat{f}]}_{L^1(\mathbb{R}^2)}\lesssim \delta^{\lambda}\log(\delta^{-1})^{c}N(\Omega, \delta)\Norm{f}_{L^1(\mathbb{R}^2)}.
\end{align*}
\end{customprop}

\begin{customprop}{F}[\cite{sz}]\label{propB}
Let $\Omega$ be a convex domain satisfying (\ref{8est}) and let $b\in C_0^{\infty}$ be supported in the sector $S=\{\xi:\,|\xi_1|\le 2^{-10M}|\xi_2|,\,\xi_2<0\}$. Let $\alpha\mapsto(\alpha, \gamma(\alpha))$ be the parametrization of $\partial\Omega\cap S$ as a graph, as in Lemma \ref{elemlemma}. For any subinterval $I$ of $[-1/2, 1/2]$ denote by $I^{\ast}$ the interval with the same center and with length $\frac{4}{3}|I|$. For $\delta<1/2$ let $\mathfrak{J}_{\delta}$ be the set of open subintervals $I$ of $[-1, 1]$ with the property that $|I|\ge 2^{-5M}\delta$ and 
\begin{align}\label{leftest}
(t-s)(\gamma_L^{\prime}(t)-\gamma_R^{\prime}(s))\le 2^{5}\delta\text{ for }s<t,\,s,t,\in I^{\ast}.
\end{align}
Let $\mathfrak{B}$ be the set of $C^2$ functions $\beta$ supported on $(-1/2, 1/2)$ so that
\begin{align*}
|\beta^{(k)}(t)|\le 1, \qquad k=0, \ldots, 4.
\end{align*}
Suppose $I=(c_I-|I|/2, c_I+|I|/2)\in\mathfrak{J}_{\delta}$. Let
\begin{align}\label{weirdmult}
m_{\beta_1, \beta_2, I}(\xi)=b(\xi)\beta_1(\frac{\delta^{-1}}{2}(1-\rho(\xi)))\beta_2(|I|^{-1}(\xi_1-c_I))
\end{align}
where $\beta_1, \beta_2\in\mathfrak{B}$. Then for any $\beta_1, \beta_2\in\mathfrak{B}$ and $I\in\mathfrak{J}_{\delta}$, 
\begin{align}\label{multest}
\Norm{F^{-1}[m_{\beta_1, \beta_2, I}]}_1\lesssim\log(\delta^{-1}).
\end{align}
Let
\begin{align*}
\mathfrak{M}_{\delta}f(x)=\sup_{\beta_1, \beta_2\in\mathfrak{B}}\sup_{I\in\mathfrak{J}_{\delta}}\big||\mathcal{F}^{-1}[m_{\beta_1, \beta_2, I}]|\ast f(x)\big|
\end{align*}
and let
\begin{align*}
\overline{M}_{\delta}f(x)=\sup_{x\in R\in \mathcal{C}_{\delta}}\frac{1}{|R|}\int_R|f(y)|\,dy,
\end{align*}
where 
\begin{align*}
\mathcal{C}_{\delta}=\{R:\, R\text{ is a rectangle of dimensions }\delta\times (a_{j+1}-a_j)\\
\text{ with longer side of slope }\gamma_L^{\prime}(a_j),\text{ where }a_j, a_{j+1}\in\mathfrak{A}(\delta)\}.
\end{align*}
Then
\begin{align}\label{maxest}
\mathfrak{M}_{\delta}f(x)\lesssim \log(\delta^{-1})\overline{M}_{\delta}f(x).
\end{align}

\end{customprop}

\section{Construction of $\Omega$ and some algebraic disjointness lemmas}\label{constrsec}
We will now construct domains which we will show satisfy the statement of Theorem \ref{2ndmainprop}. The idea is to construct a convex domain $\Omega$ such that the kernels of the pieces of the multiplier obtained by decomposing the multiplier as in Proposition \ref{propB} exhibit a high degree of cancellation with each other. In \cite{sz}, it was shown that for abitrary convex domains that the supports of the convolution of pairs of pieces of the multiplier were more or less disjoint. This was used to prove the endpoint $p=4/3$ estimate using duality and an $L^4$ argument similar to C\'{o}rdoba's treatment of the (spherical) Bochner-Riesz means in $\mathbb{R}^2$ (see \cite{cor}). Here, we construct a domain so that the supports of the $m$-fold convolution of $m$-tuples of pieces of the mutiplier are more or less disjoint, which we will use to prove an $L^{2m}$ estimate in the same vein as in \cite{cor} and \cite{sz}.
\newline
\indent
Before constructing $\Omega$, we will need the following basic lemma.

\begin{lemma}\label{disjlemma}
For any integer $N>10$ and any integer $m\ge 1$, there exists a collection $\mathcal{I}$ of $N$ disjoint subintervals of $[-\frac{1}{2}, \frac{1}{2}]$ each of size $\frac{N^{-(2m-1)}}{3m}$ so that 
\begin{align*}
\{I_1+I_2+\cdots+I_{m}\}_{I_1,\ldots, I_{m}\in\mathcal{I}}
\end{align*} 
is a pairwise disjoint collection. 
\end{lemma}

\begin{proof}[Proof of Lemma \ref{disjlemma}]
Let $M$ be an integer strictly less than $N$. We will show that if $\mathcal{I}_M$ is a collection of $M$ disjoint subintervals of $[-\frac{1}{2}, \frac{1}{2}]$ each of size $\frac{N^{-(2m-1)}}{3m}$ satisfying the algebraic disjointness condition of the lemma, then there is a collection $\mathcal{I}_{M+1}$ of $M+1$ disjoint subintervals of $[-\frac{1}{2},\frac{1}{2}]$ of size $\frac{N^{-(2m-1)}}{3m}$ satisfying the same condition. 
\newline
\indent
Indeed, suppose that such a collection $\mathcal{I}_M$ exists. Suppose $I_1, \ldots, I_{m}\in\mathcal{I}_M$. Then given any collection of $m-1$ intervals $I_{m+1}, \ldots I_{2m-1}\in\mathcal{I}_M$, there is an interval $I_{(I_1, \ldots, I_{2m-1})}\subset[-\frac{1}{2}, \frac{1}{2}]$ of width no larger than $\frac{2N^{-(2m-1)}}{3}$ such that
\begin{align*}
(I_1+\cdots+I_{m})-(I_{m+1}+\cdots+I_{2m-1})\subset I_{(I_1, \ldots, I_{2m-1})}.
\end{align*}
Now define
\begin{align*}
E=\bigcup_{(I_1, \ldots, I_{2m-1})\in(\mathcal{I}_M)^{2m-1}}I_{(I_1, \ldots, I_{2m-1})}.
\end{align*}
Then since $\text{card}((\mathcal{I}_M)^{2m-1})=M^{2m-1}$, we have $|E|\le\frac{2}{3}\cdot(\frac{M}{N})^{2m-1}$. Since $M<N$, we have $|[-\frac{1}{2}, \frac{1}{2}]\setminus E|\ge\frac{1}{3}.$ Since $E$ is a union of no more than $M^{2m-1}$ disjoint intervals, the average gap length between consecutive disjoint intervals in $E$ is at least $\frac{1}{6}M^{-7}\ge\frac{1}{6}N^{-7}$. Thus there exists an interval $I$ of length $\frac{N^{-(2m-1)}}{3m}$ such that $I\subset [-\frac{1}{2}, \frac{1}{2}]\setminus E$. Now set $\mathcal{I}_{M+1}=\mathcal{I}_M\cup\{I\}$. Then $\mathcal{I}_{M+1}$ is a collection of $M+1$ disjoint subintervals of $[-\frac{1}{2}, \frac{1}{2}]$ each of size $\frac{N^{-(2m-1)}}{3m}$ satisfying the algebraic disjointness condition of the lemma. By induction on $M$, the proof is complete.
\end{proof}

\subsection*{Construction of $\Omega$}
We now proceed to construct the convex domain $\Omega$ which we will show satisfies the statement of Theorem \ref{2ndmainprop} with $\kappa_{\Omega}=\frac{1}{4m-2}$. It will then be easy to explain how to modify the construction to produce a domain which satisfies the statement of Theorem \ref{2ndmainprop} with $\kappa_{\Omega}\in [0, \frac{1}{4m-2})$. 
\newline
\indent
For each integer $k\ge 0$, we inductively define a collection $\mathcal{I}_k$ of disjoint subintervals of $[-\frac{1}{2}, \frac{1}{2}]$. We set $\mathcal{I}_0=\{[-\frac{1}{2}, \frac{1}{2}]\}$. For each $k\ge 0$, we define $\mathcal{I}_{k+1}$ to be a collection of $2^{k+4}\cdot\text{card}(\mathcal{I}_k)$ subintervals of intervals in $\mathcal{I}_k$ obtained by applying Lemma \ref{disjlemma} with $N=2^{k+4}$ to each interval of $\mathcal{I}_k$. More precisely, if we let $\tilde{\mathcal{I}}_k$ be a collection of $N$ disjoint subintervals of $[-\frac{1}{2}, \frac{1}{2}]$ each of size $\frac{N^{-(2m-1)}}{3m}$ given by Lemma \ref{disjlemma} with $N=2^{k+4}$, then for each $I\in\tilde{\mathcal{I}}_k$, let $\tilde{\mathcal{I}}_{k, I}$ be the rescaling of $\tilde{\mathcal{I}}_k$ to $I$, that is, if the endpoints of $I$ are $a$ and $b$ with $a<b$, set $\tilde{\mathcal{I}}_{k, I}=a+(b-a)\tilde{\mathcal{I}}_k$. Then set 
\begin{align*}
\mathcal{I}_{k+1}=\bigcup_{I\in\mathcal{I}_k}\tilde{\mathcal{I}}_{k, I}.
\end{align*}
For each $k$, define $S_k$ to be the set of endpoints of intervals in $\mathcal{I}_k$, and define $\Omega_k$ to be the convex polygon with vertices at 
\begin{align*}
\{(x-\frac{1}{2}, x^2-8):\, x\in S_k\}\cup\{(-8, 0); (-8, 8); (8, 0); (8, 8)\}.
\end{align*}
Let $\Omega$ be the convex domain so that $\partial\Omega$ is the uniform limit of $\{\partial\Omega_k\}$ as $k\to\infty$. Note that $\Omega$ satisfies (\ref{8est}) with $M=10$.

\begin{lemma}\label{ultlemma}
Let $\Omega$ be constructed as described previously. For every $\delta>0$, there exist integer constants $C_1(\delta)$, $C_2(\delta)$ with $C_1(\delta)=O(\delta^{-\epsilon})$ and $C_2(\delta)=O(\delta^{-\epsilon})$ for every $\epsilon>0$ so that if $\mathcal{J}_{\delta}$ denotes the collection of $Q(\Omega, \delta)$ essentially disjoint intervals obtained from the decomposition of $[-1, 1]$ as described in Section \ref{prelim}, then we can write
\begin{align*}
\mathcal{J}_{\delta}=\bigcup_{l=1}^{C_1(\delta)}\mathcal{J}_{\delta, l}
\end{align*}
such that for each $l$, no point of $\mathbb{R}$ is contained in more than $C_2(\delta)$ of the sets
\begin{align*}
\{I_1+\cdots +I_m\}_{I_1, \ldots, I_m\in\mathcal{J}_{\delta, l}}.
\end{align*}
In particular, this implies that $\mathcal{E}_m(\partial\Omega)=0$.
\end{lemma}

\begin{proof}[Proof of Lemma \ref{ultlemma}]
Given $\delta>0$, let $K(\delta)$ be the largest integer such that each interval in $\mathcal{I}_{K(\delta)}$ has size $\ge\delta^{1/2}$. For each integer $k\ge 0$, let $\mathcal{I}_{k}^{\prime}$ denote the set of essentially disjoint subintervals corresponding to the decomposition of $[-1/2, 1/2]$ given by the partition $S_k$ of $[-1/2, 1/2]$. Then for every $\delta>0$, each element of $\mathcal{J}_{\delta}$ intersects no more than $10$ elements of $\mathcal{I}_{K(\delta)}^{\prime}$, and each element of $\mathcal{I}_{K(\delta)}^{\prime}$ intersects no more than $10$ elements of $\mathcal{J}_{\delta}$. Moreover, all but at most $10$ elements of $\mathcal{J}_{\delta}$ are covered by a union of elements of $\mathcal{I}_{K(\delta)}^{\prime}$. It thus suffices to prove the lemma with $\mathcal{J}_{\delta}$ replaced by $\mathcal{I}_{K(\delta)}^{\prime}$. 
\newline
\indent
It is easy to compute that $K(\delta)\lesssim(\log(\delta^{-1}))^{1/2}=O(\delta^{-\epsilon})$ for every $\epsilon>0$. We organize $\mathcal{I}_{K(\delta)}^{\prime}$ into $K(\delta)+1$ disjoint subcollections as follows. Set $(\mathcal{I}_{K(\delta)}^{\prime})_0=\mathcal{I}_{K(\delta)}$. Set $(\mathcal{I}_{K(\delta)}^{\prime})_1=\mathcal{I}_{1}^{\prime}\setminus\mathcal{I}_1$ and for $1<k\le K(\delta)-1$ inductively define 
\begin{align*}
(\mathcal{I}_{K(\delta)}^{\prime})_{k+1}=\mathcal{I}_k^{\prime}\setminus(\mathcal{I}_k\cup(\mathcal{I}_{K(\delta)}^{\prime})_{k}).
\end{align*} 
Then
\begin{align*}
\mathcal{I}_{K(\delta)}^{\prime}=\bigsqcup_{k=0}^{K(\delta)}(\mathcal{I}_{K(\delta)}^{\prime})_k.
\end{align*}
It is also easy to see that for $k>1$, every element of $(\mathcal{I}_{K(\delta)}^{\prime})_k$ is a subset of an element of $\mathcal{I}_{k-1}$. In fact, we can think of $(\mathcal{I}_{K(\delta)}^{\prime})_k$ for $k>0$ as the ``gaps" leftover after subdividing $\mathcal{I}_{k-1}$.
\newline
\indent
We now show that for any $k\ge 0$, no point of $\mathbb{R}$ is contained in more than $(m!)^k$ of the sets $\{I_1+\cdots+I_m\}_{I_1, \ldots, I_m\in\mathcal{I}_k}$. We prove this by induction on $k$. The base case is trivial. Suppose that this is true for a given $k$. Fix $x\in\mathbb{R}$, and suppose there are intervals $I_1, \ldots, I_m\in\mathcal{I}_{k+1}$ such that $x\in (I_1+\cdots+I_m)$. Then there are intervals $I_{m+1}, \ldots, I_{2m}\in\mathcal{I}_{k}$ such that $I_1\subset I_{m+1}$, $I_2\subset I_{m+2}$, \ldots, $I_{m}\subset I_{2m}$. Let us count how many $m$-tuples $(I_1^{\prime}, \ldots, I_m^{\prime})$ there are satisfying $x\in I_1^{\prime}+\cdots+I_m^{\prime}$ and $I_1^{\prime}\subset I_{m+1}, I_2^{\prime}\subset I_{m+1}^{\prime}, \ldots, I_{m}\subset I_{2m}^{\prime}$. After applying an appropriate translation and dilation, this is the same as the number of ordered $m$-tuples of intervals whose sum contains a given point, where the intervals are restricted to a collection that satisfy the properties stated in Lemma \ref{disjlemma} for some $N$. But for such a collection the number of ordered $m$-tuples is simply $m!$. By the inductive hypothesis, the number of choices of intervals $I_{m+1}, \ldots, I_{2m}\in\mathcal{I}_{k}$ is $\le (m!)^k$, and therefore the number of choices of intervals $I_1, \ldots, I_m$ is $\le (m!)^{k+1}$.
\newline
\indent
The above argument shows that no point of $\mathbb{R}$ is contained in more than $(m!)^{K(\delta)}$ of the sets $\{I_1+\cdots+I_m\}_{I_1, \ldots, I_m\in (\mathcal{I}_{K(\delta)}^{\prime})_0}$. Moreover, for every $0\le k\le K(\delta)$ no point of $\mathbb{R}$ is contained in more than $(m!)^{K(\delta)}$ of the sets $\{I_1+\cdots+I_m\}_{I_1, \ldots, I_m\in \mathcal{I}_{k}}$. Fix $k>0$, and also fix $x\in\mathbb{R}$. Given $I_1, \ldots, I_m\in\mathcal{I}_{k-1}$ with $x\in (I_1+\cdots+I_m)$, there are at most $2^{k+10}$ choices of intervals $I_{m+1}, \ldots, I_{2m}\in (\mathcal{I}_{K(\delta)}^{\prime})_k$ such that $I_1\subset I_{m+1}, I_2\subset I_{m+2}, \ldots, I_m\subset I_{2m}$. It follows that $x$ is contained in no more than $2^{K(\delta)+10}\cdot (m!)^{K(\delta)}$ of the sets $\{I_1+\cdots+I_m\}_{I_1, \ldots, I_m\in(\mathcal{I}_{K(\delta)}^{\prime})_k}$. 
\newline
\indent
As noted previously, $K(\delta)\lesssim(\log(\delta^{-1}))^{1/2}$, so $2^{K(\delta)+10}\cdot (m!)^{K(\delta)}=O(\delta^{-\epsilon})$ for every $\epsilon>0$. Thus we have proven the lemma with $C_1(\delta)=K(\delta)+1$ and $C_2(\delta)=2^{K(\delta)+10}\cdot 24^{K(\delta)}$.
\end{proof}

\begin{lemma}\label{kappalemma}
Let $\Omega$ be constructed as described previously. Then $\kappa_{\Omega}=\frac{1}{4m-2}$.
\end{lemma}

\begin{proof}[Proof of Lemma \ref{kappalemma}]
Let $K(\delta)$ be defined as in the proof of Lemma \ref{ultlemma}. $K(\delta)$ is the greatest integer such that 
\begin{align*}
\prod_{n=1}^{K(\delta)}2^{-(2m-1)(n+4)}\ge\delta^{1/2}.
\end{align*}
It follows that 
\begin{align*}
\text{card}(\mathcal{I}_{K(\delta)+1})=\prod_{n=1}^{K(\delta)+1}2^{(n+4)}\ge\delta^{-1/(4m-2)},
\end{align*}
and hence
\begin{align*}
\delta^{-1/(4m-2)}2^{-K(\delta)-4}\le \text{card}(\mathcal{I}_{K(\delta)})\le\delta^{-1/(4m-2)}.
\end{align*}
As noted previously, $2^{-K(\delta)}=O(\delta^{-\epsilon})$ for every $\epsilon>0$, and hence by Lemma \ref{covlemma}, 
\begin{multline*}
\kappa_{\Omega}=\limsup_{\delta\to 0}\frac{\log(N(\Omega, \delta))}{\log(\delta^{-1})}=\limsup_{\delta\to 0}\frac{Q(\Omega, \delta)}{\log(\delta^{-1})}
\\
=\limsup_{\delta\to 0}\frac{\text{card}(\mathcal{J}_{\delta})}{\log(\delta^{-1})}=\limsup_{\delta\to 0}\frac{\text{card}(\mathcal{I}_{K(\delta)})}{\log(\delta^{-1})}=\frac{1}{4m-2}.
\end{multline*}
\end{proof}

\begin{remark}\label{genremark}
Let $\kappa\in [0, \frac{1}{4m-2})$. We now describe how we may modify the construction of $\Omega$ so that it still satisfies the hypotheses of Lemma \ref{ultlemma}, but $\kappa_{\Omega}=\kappa$. Obviously, we may replace Lemma \ref{disjlemma} with the weaker statement that there exists $N^{c}$ (instead of $N$) disjoint subintervals satisfying the hypotheses of Lemma \ref{disjlemma} with $0\le c<1$. If we repeat the same construction of $\Omega$ described previously except applying this weaker version of Lemma \ref{disjlemma} instead, we will produce a domain $\Omega$ with $\kappa_{\Omega}=\kappa$ if we choose $c$ appropriately. Verification of the details is left to the reader.
\end{remark}

\section{Proof of Theorem \ref{genthm}}\label{thmproof}

To prove Theorem \ref{genthm} in the case that $\lambda>0$, it only remains to prove the following proposition.

\begin{proposition}\label{finalprop}
Let $\Omega$ be a convex domain in $\mathbb{R}^2$ containing the origin and let $\Theta$ be its associated set of directions. Let $n\ge 2$ be an integer. Suppose that $\mathcal{E}_n(\partial\Omega)=\alpha$ for some integer $0\le\alpha\le n\kappa_{\Omega}$ and that $\Norm{M_{\Theta, \delta}}_{L^{\frac{n}{n-1}}(\mathbb{R}^2)\to L^{\frac{n}{n-1}}(\mathbb{R}^2)}\le C_{\epsilon}\delta^{-\beta-\epsilon}$ for some $0\le\beta\le\kappa_{\Omega}(\frac{n-2}{n})$ and every $\epsilon>0$. Then if $m_{\delta, \lambda}$ is as in the statement of Proposition \ref{propA}, there is a constant $C(\delta)=O(\delta^{-\epsilon})$ for every $\epsilon>0$ such that 
\begin{align}\label{87est}
\Norm{\mathcal{F}^{-1}[m_{\delta, \lambda}\widehat{f}]}_{L^{\frac{2n}{2n-1}}(\mathbb{R}^2)}\lesssim\delta^{\lambda}C(\delta)\delta^{-\frac{\alpha}{2n}-\frac{\beta}{2}}\Norm{f}_{L^{\frac{2n}{2n-1}}(\mathbb{R}^2)}.
\end{align}

\end{proposition}
Interpolating Proposition \ref{finalprop} with Proposition \ref{propA} gives the result of Theorem \ref{genthm} for $\lambda>0$.

\begin{proof}[Proof of Proposition \ref{finalprop}]
By duality, to prove (\ref{87est}) it suffices to prove
\begin{align}\label{8est2}
\Norm{\mathcal{F}^{-1}[m_{\delta, \lambda}\widehat{f}]}_{L^{2n}(\mathbb{R}^2)}\lesssim\delta^{\lambda}C(\delta)\delta^{-\frac{\alpha}{2n}-\frac{\beta}{2}}\Norm{f}_{L^{2n}(\mathbb{R}^2)}.
\end{align}
Using an appropriate partition of unity and rotation invariance, it in fact suffices to show that if $b\in C_0^{\infty}$ is as in the statement of Proposition \ref{propB}, then
\begin{align}\label{8est3}
\Norm{\mathcal{F}^{-1}[b\cdot m_{\delta, \lambda}\widehat{f}]}_{L^{2n}(\mathbb{R}^2)}\lesssim\delta^{\lambda}C(\delta)\delta^{-\frac{\alpha}{2n}-\frac{\beta}{2}}\Norm{f}_{L^{2n}(\mathbb{R}^2)}.
\end{align}
Let $\mathcal{J}_{\delta}$ denote the collection of $Q(\Omega, \delta)$ essentially disjoint intervals obtained from the decomposition of $[-1, 1]$ as described in Section \ref{prelim}. For each $I=(\alpha_0, \alpha_1)\in\mathcal{J}_{\delta}$, set $B(I)$ to be a rectangle that has one side parallel to $(1, \gamma^{\prime}(\alpha_0))$, contains $\text{supp}(b\cdot m_{\delta, \lambda})\cap\{x:\,x_1\in I\}$, and such that its $1/2$-dilate is contained in $\text{supp}(b\cdot m_{\delta, \lambda})\cap\{x:\,x_1\in I\}$. Since $\mathcal{E}_n(\partial\Omega)=\alpha$, there are constants $C_1(\delta)$ and $C_2(\delta)$ such that $C_1(\delta)^{2n}C_2(\delta)=O(\delta^{-\alpha-\epsilon})$ for every $\epsilon>0$, and such that we may write $J_{\delta}=\bigcup_{l=1}^{C(\delta)}\mathcal{J}_{\delta, l}$ so that for each $l$, no point of $\mathbb{R}^2$ is contained in more than $C_2(\delta)$ of the sets 
\begin{align*}
\{B(I_1)+\cdots+B(I_{n})\}_{I_j\in\mathcal{J}_{\delta, l}}.
\end{align*}
Now let $\mathfrak{J}_{\delta}$ be defined as in the statement of Proposition \ref{propB}, and let $\{\beta_i\}$ be a partition of unity of $[-\frac{1}{4}, \frac{1}{4}]$ satisfying
\begin{enumerate}
\item $\sum_i\beta_i$ is supported in $(-\frac{1}{2}, \frac{1}{2})$, 
\item Every $\beta_i$ is of the form $\beta(|I|^{-1}(\cdot-c_I))$ for some $\beta\in\mathfrak{B}$ and for some interval $I\in\mathfrak{J}_{\delta}$ with center $c_I$,
\item Each interval in $\mathcal{J}_{\delta}$ intersects the support of at most $(\log(\delta^{-1}))^2$ of the $\beta_i$'s,
\item If the support of $\beta_i$ intersects some $I\in\mathcal{J}_{\delta}$ then the support of $\beta_i$ is contained in $10 I$, where the dilation is taken from the center of $I$.

\end{enumerate}
Set $m_i(\xi)=\beta_i(\xi_1)b(\xi)m_{\delta, \lambda}(\xi)$, and define an operator $T_i$ by 
\begin{align*}
T_if(x)=\delta^{-\lambda}\mathcal{F}^{-1}[m_i\widehat f](x). 
\end{align*}
Set
\begin{align*}
\mathfrak{I}_1=\{i: \text{supp}(\beta_i)\cap(\cup\mathcal{J}_{\delta, 1})\ne\emptyset\},
\end{align*}
and for $l=2, \ldots, C_1(\delta)$, set 
\begin{align*}
\mathfrak{I}_l=\{i:\,\text{supp}(\beta_i)\cap (\cup\mathcal{J}_{\delta, l-1})=\emptyset\text{ and }\text{supp}(\beta_i)\cap (\cup\mathcal{J}_{\delta, l})\ne\emptyset\}.
\end{align*}
We write
\begin{align*}
\sum_iT_if(x)=\sum_{l=1}^{C_1(\delta)}\sum_{i\in \mathfrak{I}_l}T_if(x).
\end{align*}
We now proceed with an argument similar to the familiar one from \cite{cor}. Using the triangle inequality, H\"{o}lder's inequality and Plancherel, we have
\begin{multline}\label{step1}
\Norm{\sum_iT_if}_{2n}^{2n}\lesssim\bigg(\sum_{l=1}^{C_1(\delta)}\Norm{\sum_{i\in\mathfrak{I}_l}T_if}_{2n}\bigg)^{2n}\lesssim C_1(\delta)^{2n-1}\sum_{l=1}^{C_1(\delta)}\Norm{\sum_{i\in\mathfrak{I}_l}T_if}_{2n}^{2n}
\\
\lesssim C_1^{2n-1}(\delta)\sum_{l=1}^{C_1(\delta)}\int_{\mathbb{R}^2}\bigg|\sum_{i\in\mathfrak{I}_l} T_if(x)\bigg|^{2n}\,dx
\\
\lesssim C_1(\delta)^{2n-1}\sum_{l=1}^{C_1(\delta)}\int \bigg|\sum_{i_1, \ldots, i_{n}\in\mathfrak{I}_l}T_{i_1}f(x)T_{i_2}f(x)\cdots T_{i_{n}}f(x)\bigg|^2\,dx\
\\
\lesssim C_1(\delta)^{2n-1}\sum_{l=1}^{C_1(\delta)}\int\bigg|\sum_{i_1, \ldots, i_{n}\in\mathfrak{I}_l}\widehat{T_{i_1}f}\ast\widehat{T_{i_2}f}\ast\cdots\ast\widehat{T_{i_{n}}f}(\xi)\bigg|^2\,d\xi.
\end{multline}
Now note that no point of $\mathbb{R}^2$ is contained in more than $C_2(\delta)$ of the sets 
\begin{align*}
\bigg\{\text{supp}(\widehat{T_{i_1}f}\ast\widehat{T_{i_2}f}\ast\cdots\ast\widehat{T_{i_{n}}f}(\xi))\bigg\}_{i_1,\ldots, i_{n}\in\mathfrak{I}_l}. 
\end{align*}
Set $C_3(\delta)=C_1(\delta)^{2n-1}C_2(\delta)(\log(\delta^{-1}))^3$. It follows that the right hand side of (\ref{step1}) is bounded by a constant times
\begin{align}\label{step2}
C_3(\delta)\sum_{l=1}^{C_1(\delta)}\int\sum_{i_1, \ldots, i_{n}\in\mathfrak{I}_l}|\widehat{T_{i_1}f}\ast\widehat{T_{i_2}f}\ast\cdots\ast\widehat{T_{i_{n}}f}(\xi)|^2\,d\xi,
\end{align}
and by Plancherel, (\ref{step2}) is equal to
\begin{multline}\label{step3}
C_3(\delta)\sum_{l=1}^{C_1(\delta)}\int\sum_{i_1,\ldots, i_{n}\in\mathfrak{I}_l}|T_{i_1}f(x)T_{i_2}f(x)\cdots T_{i_{n}}f(x)|^2\,dx
\\
\lesssim
C_3(\delta)\sum_{l=1}^{C_1(\delta)}\int\sum_{i_1, \ldots, i_{n}\in\mathfrak{I}_l}|T_{i_1}f(x)T_{i_2}f(x)\cdots T_{i_{n}}f(x)|^2\,dx
\\
\lesssim
C_3(\delta)\sum_{l=1}^{C_1(\delta)}\int\bigg(\sum_{i\in\mathfrak{I}_l}|T_{i}f(x)|^2\bigg)^{n}\,dx.
\\
\end{multline}
Let $\phi:[-2, 2]\to\mathbb{R}$ be a smooth function identically $1$ on $[-1, 1]$. For each $i$, write $\beta_i=\beta(|I|^{-1}(\cdot-c_I))$ for some $\beta\in\mathfrak{B}$ and set $\psi_i(\xi)=\phi(|I|^{-1}(\xi_1-c_I))$. Define a multiplier operator $S_i$ by
\begin{align*}
S_if=\mathcal{F}^{-1}[\psi_i\widehat f].
\end{align*}
If $K_i$ denotes the convolution kernel of the operator $T_i$, let $\tilde{T_i}$ be the operator with convolution kernel $|K_i|$. By duality, the right hand side of (\ref{step3}) is bounded by
\begin{multline}\label{step4}
C_3(\delta)\sum_{l=1}^{C_1(\delta)}\bigg(\sup_{\Norm{w}_{{\frac{n}{n-1}}}\le 1}\int\sum_{i\in\mathfrak{I}_l}|T_if(x)|^2w(x)\,dx\bigg)^{n}
\\
\lesssim
C_3(\delta)\sum_{l=1}^{C_1(\delta)}\bigg(\sup_{\Norm{w}_{\frac{n}{n-1}}\le 1}\int\sum_{i\in\mathfrak{I}_l}|S_if(x)|^2(\sup_i|\tilde{T}_iw(x)|)\,dx\bigg)^{n}
\\
\lesssim
C_3(\delta) \sum_{l=1}^{C_1(\delta)}\Norm{\bigg(\sum_{i\in\mathfrak{J}_l}|S_if(x)|^2\bigg)^{1/2}}_{2n}^{2n}\sup_{\Norm{w}_{\frac{n}{n-1}}\le 1}\Norm{\sup_i|\tilde{T}_iw|}_{\frac{n}{n-1}}^{n}.
\end{multline}
By (\ref{multest}) and the assumption $\Norm{M_{\Theta(\Omega), \delta}}_{L^{\frac{n}{n-1}}(\mathbb{R}^d)\to L^{\frac{n}{n-1}}(\mathbb{R}^d)}=O_{\epsilon}(\delta^{-\beta-\epsilon})$, we have
\begin{align}\label{inter3}
\Norm{\sup_i|\tilde{T}_if|}_{\frac{n}{n-1}}\lesssim C(\delta)\delta^{-\beta}\Norm{f}_{\frac{n}{n-1}}
\end{align}
where $C(\delta)=O(\delta^{-\epsilon})$ for every $\epsilon>0$. Moreover, since the supports of the $\psi_i$ are $\lesssim\log(\delta^{-1})$-disjoint, by Rubio de Francia's theorem on square functions for arbitrary collections of intervals \cite{rubio},
we have
\begin{align}\label{square}
\Norm{\bigg(\sum_{i\in\mathfrak{J}_l}|S_if(x)|^2\bigg)^{1/2}}_{2n}\lesssim\log(\delta^{-1})\Norm{f}_{2n}.
\end{align}
Set $C_4(\delta)=C(\delta)^{2n}C_1(\delta)C_3(\delta)\log(\delta^{-1})^{4n}\delta^{-\beta n}$. By (\ref{step4}), (\ref{inter3}) and (\ref{square}), we have
\begin{align}\label{step5}
\Norm{\sum_iT_if}_{2n}\lesssim_{\epsilon}\delta^{-\epsilon} (C_4(\delta))^{1/2n}\Norm{f}_{2n},
\end{align}
and since $C_4(\delta)\lesssim_{\epsilon} C_1(\delta)^{2n}C_2(\delta)\delta^{-\beta n-\epsilon}\lesssim_{\epsilon}\delta^{-\alpha-\beta n-\epsilon}$ for every $\epsilon>0$, this proves (\ref{8est3}) and thus completes the proof of Proposition \ref{finalprop}.
\end{proof}

It only remains to prove Theorem \ref{2ndmainprop} in the case that $\lambda=0$. This will follow fairly easily from Bateman's characterization in \cite{bateman} of all planar sets of directions which admit Kakeya sets and Fefferman's proof in \cite{feff2} that the ball multiplier is unbounded on $L^p(\mathbb{R}^2)$ for $p\ne 2$.

\begin{proof}[Proof of Theorem \ref{2ndmainprop} in the case that $\lambda=0$]
Let $\Theta$ denote the set of all directions associated to $\Omega$. We claim that if $\Theta$ is a union of finitely many lacunary sets of finite order, then $\kappa_{\Omega}=0$. Indeed, suppose that $\Theta$ is a union of $N_1$ lacunary sets of order $N_2$. Then it is easy to see that there is a subset of $\Theta_{\delta}\subset\Theta$ of cardinality $\le N_1(\log(\delta^{-1}))^{N_2}$ such that every element of $\Theta$ is contained in a $\delta$ neighborhood of an element of $\Theta_{\delta}$. It follows that $N(\Omega, \delta)\lesssim N_1(\log(\delta^{-1}))^{N_2}$, and hence $\kappa_{\Omega}=0$.
\newline
\indent
We say that $\Theta$ \textit{admits Kakeya sets} if for each postive integer $N$ there is a collection $\mathcal{R}_{\Theta}^{(N)}$ of rectangles with longest side parallel to a direction in $\Theta$ so that 
\begin{align*}
\bigg|\bigcup_{R\in\mathcal{R}_{\Theta}^{(N)}}R\bigg|\le\frac{1}{N}\bigg|\bigcup_{R\in\mathcal{R}_{\Theta}^{(N)}}\tilde{R}\bigg|,
\end{align*}
where $\tilde{R}$ denotes the rectangle with the same center and width as $R$ but with three times the length. In \cite{bateman}, the following theorem was proved.
\begin{customthm}{G}[Bateman, \cite{bateman}]\label{kakeyaprop}
Fix $1<p<\infty$. The following are equivalent:
\begin{enumerate}
\item $M_{\Theta}$ is bounded on $L^p(\mathbb{R}^2)$;
\item $\Theta$ does not admit Kakeya sets;
\item there exist $N_1, N_2<\infty$ such that $\Theta$ is covered by $N_1$ lacunary sets of order $N_2$.
\end{enumerate}
\end{customthm}
It follows from Theorem \ref{kakeyaprop} that if $\kappa_{\Omega}>0$, then $\Theta$ admits Kakeya sets. We will now show that if $\kappa_{\Omega}>0$, then $T_{0}$ is unbounded on $L^p$ for all $p\ne 2$. Assume that $T_0$ is bounded on $L^p$ for some $p>2$. Let $\{v_j\}$ be a sequence of unit vectors parallel to directions in $\Theta$, and let $H_j$ denote the half-plane $\{x\in\mathbb{R}^2:\,x\cdot v_j\ge 0\}$. For each $j$, define an operator $T_j$ by
\begin{align*}
\mathcal{F}[T_jf](\xi)=\chi_{H_j}(\xi)\widehat f(\xi).
\end{align*}
Then arguing as in \cite{feff2}, there is an absolute constant $C$ (independent of the choice of the sequence $\{v_j\}$) such that
\begin{align*}
\Norm{(\sum_j|T_jf_j|^2)^{1/2}}_p\le C\Norm{(\sum_j|f_j|^2)^{1/2}}_p.
\end{align*}
Since $\Theta$ admits Kakeya sets, for each $\upsilon>0$ we may choose a sequence of unit vectors $\{v_j\}$ parallel to directions in $\Theta$ such that there is a collection of rectangles $\{R_j\}$ with the longest side of $R_j$ parallel to $v_j$ and so that
\begin{align*}
\bigg|\bigcup_j R_j|\le\upsilon\bigg|\bigcup_j\tilde{R}_j\bigg|.
\end{align*}
Let $E:=\bigcup_j R_j$ and let $E^{\prime}:=\bigcup_j\tilde{R}_j$. Then arguing as in \cite{feff}, we have
\begin{align*}
\int_E\sum_j|T_j \chi_{R_j}(x)|^2\,dx\gtrsim\sum_j|E\cap\tilde{R}_j|\gtrsim |E^{\prime}|,
\end{align*}
but by H\"{o}lder's inequality
\begin{align*}
\int_E\sum_j|T_j\chi_{R_j}(x)|^2\,dx\lesssim |E|^{(p-2)/p}(\sum_j|R_j|)^{2/p}\lesssim\upsilon^{(p-2)/p}|E^{\prime}|.
\end{align*}
Letting $\upsilon\to 0$ gives a contradiction.
\end{proof}

\section{Lower bounds using Khinchine's inequality}\label{khinsec}

In this section, we will prove Theorem \ref{khin1}, which gives lower bounds on the range of $\lambda$ for which $T_{\lambda}$ is bounded on $L^p$ for general convex domains with a given value of $\kappa_{\Omega}$. To prove Theorem \ref{khin1}, we first show that boundedness of $T_{\lambda}$ on $L^p$ implies (\ref{tdelta}), where $T_{\lambda}^{\delta}$ is defined below. We then test $T_{\lambda}^{\delta}$ on randomly defined functions and apply Khinchine's inequality to estimate the $L^1$ norm of these functions. After applying $T_{\lambda}^{\delta}$, the randomness of these test functions will effectively ``disappear" due to the test functions being essentially constant on a sequence of disjoint caps in $\mathcal{B}_{\delta}$. The $L^2$ mapping properties of $T^{\lambda}$ acting on these functions will be easy to quantify using Plancherel. The last step is simply to interpolate between $L^1$ and $L^2$.

\begin{proof}[Proof of Theorem \ref{khin1}]
Suppose that $T_{\lambda}$ is bounded on $L^p$. Let $\phi\in C_0^{\infty}(\mathbb{R})$ be supported in $[-2, 2]$ and identically $1$ on $[-1, 1]$. Let 
\begin{align*}
m_{\delta}(s)=\phi(\delta^{-1}(1-s))
\end{align*} 
and let $T_{\lambda}^{\delta}$ be the operator defined by
\begin{align*}
\mathcal{F}[T_{\lambda}^{\delta}f](\xi)=m_{\delta}(\rho(\xi))\widehat f(\xi).
\end{align*}
We will use the well-known subordination formula
\begin{align}\label{subord}
m(\rho)=\frac{(-1)^{\floor{\lambda}+1}}{\Gamma(\lambda+1)}\int_0^{\infty}s^{\lambda}m^{(\lambda+1)}(s)(1-\frac{\rho}{s})_+^{\lambda}\,ds,
\end{align}
where
\begin{align*}
\widehat{m^{(\gamma)}}(\tau)=(-1)^{\floor{\gamma}}(-i\tau)^{\gamma}\widehat m(\tau).
\end{align*}
See \cite{trebels} for a proof of (\ref{subord}). Together, (\ref{subord}) and the $L^p$-boundedness of $T_{\lambda}$ imply that
\begin{align}\label{tdelta}
\Norm{T_{\lambda}^{\delta}}_{L^p\to L^p}\lesssim \delta^{-\lambda}.
\end{align}
Let $\mathcal{J}_{\delta}$ denote the collection of $Q(\Omega, \delta)\lesssim\log(2+\delta^{-1})N(\Omega, \delta)$ essentially disjoint intervals obtained from the decomposition of $[-1, 1]$ into intervals with endpoints in $\mathfrak{A}(\delta)=\{a_0, \ldots, a_Q\}$ as described in Section \ref{prelim}. By rotation invariance, we may assume without loss of generality that $Q(\Omega, \delta)\gtrsim N(\Omega, \delta)$. For each $0\le j\le Q-1$, let $c_j=\frac{a_j+a_{j+1}}{2}$. Now observe that for $\delta$ sufficiently small there must be $\gtrsim N(\Omega, \delta)$ indices $j$ such that $a_{j+1}-a_j\le N(\Omega, \delta)^{-1}\log(\delta^{-1})$. Thus by the pigeonhole principle there is an integer $r\ge\floor{\log(N(\Omega, \delta)\log(\delta^{-1})^{-1})}$ such that there are $\gtrsim N(\Omega, \delta)(\log(\delta^{-1}))^{-1}$ indices $j$ such that $2^{-r-1}\le a_{j+1}-a_j\le 2^{-r}$. Enumerate these indices as $j_1<j_2<\cdots<j_{Q^{\prime}}$.
\newline
\indent
Let $\chi_0\in C_{0}^{\infty}(\mathbb{R})$ with $\chi\ge 0$, $\chi\equiv 1$ on $[-1, 1]$ and $\chi$ supported in $[-2, 2]$. Set $\chi(\xi_1, \xi_2)=\chi_0(\xi_1)\chi_0(\xi_2)$. Then $|\mathcal{F}[\chi](x)|\lesssim (1+|x|)^{-2}$ and $|\mathcal{F}[\chi](x)|\ge 1/2$ for $x\in B_{\frac{1}{100}}(0)$. Let $\{\epsilon_i\}$ be i.i.d. random variables with $P(\epsilon_i=\pm 1)=\frac{1}{2}$ for every $i$. Let
\begin{align*}
\psi_{\delta}(x)=\mathcal{F}[\sum_{i\equiv 0\mod(\floor{\log(\delta^{-1})})}\epsilon_i\chi(2^r(\cdot-(c_{j_i}, \gamma(c_{j_i})))](x).
\end{align*}
By Plancherel,
\begin{align}\label{1stl2}
\Norm{\psi_{\delta}}_{2}\lesssim\bigg(N(\Omega, \delta)2^{-2r}\bigg)^{1/2}
\end{align}
and
\begin{align}\label{2ndl2}
\Norm{T_{\lambda}^{\delta}\psi_{\delta}}_2\gtrsim \bigg(N(\Omega,\delta)(\log(\delta^{-1}))^{-1}2^{-r}\delta\bigg)^{1/2}.
\end{align}
By Khinchine's inequality, 
\begin{align}\label{1stl1}
\mathbb{E}[\Norm{\psi_{\delta}}_{1}]\approx Q^{\prime\frac{1}{2}}\lesssim\bigg(\log(2+\delta^{-1})N(\Omega, \delta)\bigg)^{1/2}.
\end{align}
Interpolating (\ref{1stl2}) and (\ref{1stl1}) yields
\begin{align}\label{1stlp}
\Norm{\psi_{\delta}}_p\lesssim \log(\delta^{-1})^{\frac{2-p}{2p}}N(\Omega, \delta)^{1/2}2^{-r(\frac{2p-2}{p})}, \qquad 1\le p\le 2.
\end{align}
We now prove a lower bound for $\Norm{T_{\lambda}^{\delta}\psi_{\delta}}_1$ uniformly in the realization of the random variables $\{\epsilon_i\}$. Using homogeneous coordinates, i.e. polar coordinates associated to $\Omega$, we write
\begin{multline*}
T_{\lambda}^{\delta}\psi_{\delta}(x)=\frac{1}{(2\pi)^2}\sum_{i\equiv 0\mod(\floor{\log(\delta^{-1})})}\epsilon_i\int\int\phi(\delta^{-1}(1-s))\chi_0(2^r(s\alpha-c_{j_i}))
\\
\times s(\alpha\gamma^{\prime}(\alpha)-\gamma(\alpha))e^{is(x_1\alpha+x_2\gamma(\alpha))}\,d\alpha\,ds.
\end{multline*}
Now note that for each $i$ and for $\alpha$ in the support of $\chi_0(2^r(s\alpha-c_{j_i}))$ we have
\begin{multline*}
e^{is(x_1\alpha+x_2\gamma(\alpha))}=
\\
\exp\bigg(is(x_1c_{j_i}+x_2\gamma(c_{j_i}))+is(\alpha-c_{j_i})(x_1+x_2\gamma^{\prime}(c_{j_i}))\bigg)+O(\delta|x|)
\end{multline*}
and
\begin{align*}
|\alpha\gamma^{\prime}(\alpha)-\gamma(\alpha)-c_{j_i}\gamma^{\prime}(c_{j_i})+\gamma(c_{j_i})|=O(2^{-r}).
\end{align*}
It follows that
\begin{multline*}
T_{\lambda}^{\delta}\psi_{\delta}(x)=\frac{1}{(2\pi)^2}\sum_{i\equiv 0\mod(\floor{\log(\delta^{-1})})}\epsilon_i\int\int\phi(\delta^{-1}(1-s))\chi_0(2^r(s\alpha-c_{j_i}))
\\
\times s(c_{j_i}\gamma^{\prime}(c_{j_i})-\gamma(c_{j_i}))e^{is(x_1c_{j_i}+x_2\gamma(c_{j_i}))+is(\alpha-c_{j_i})(x_1+x_2\gamma^{\prime}(c_{j_i}))}\,d\alpha\,ds 
\\
+ O(2^{-2r}\delta)+O(2^{-r}\delta^2|x|).
\end{multline*}
Rearranging this, we have
\begin{multline*}
T_{\lambda}^{\delta}\psi_{\delta}(x)=\frac{1}{(2\pi)^2}\sum_{i\equiv 0\mod(\floor{\log(\delta^{-1})})}\epsilon_i\int s\phi(\delta^{-1}(1-s))\bigg(\int\chi_0(2^r(s\alpha-c_{j_i}))
\\
\times (c_{j_i}\gamma^{\prime}(c_{j_i})-\gamma(c_{j_i}))e^{is\alpha(x_1+x_2\gamma^{\prime}(c_{j_i}))}\,d\alpha\bigg)
\\
\times e^{is(x_1c_{j_i}+x_2\gamma(c_{j_i})-c_{j_i}(x_1+x_2\gamma^{\prime}(c_{j_i})))}\,d\alpha\,ds + O(2^{-2r}\delta)+O(2^{-r}\delta^2|x|).
\end{multline*}
Set $\beta_{i}=c_{j_i}\gamma^{\prime}(c_{j_i})-\gamma(c_{j_i})$. Note that $\beta_i\approx 1$ for all $i$. We may rewrite this as
\begin{multline*}
T_{\lambda}^{\delta}\psi_{\delta}(x)=\sum_{i\equiv 0\mod(\floor{\log(\delta^{-1})})}\epsilon_i\beta_i2^{-r}\widehat{\chi_0}(-2^{-r}(x_1+x_2\gamma^{\prime}(c_{j_i})))
\\
\times\delta\cdot\widehat\phi(\delta(x_1c_{j_i}+x_2\gamma(c_{j_i})))e^{i(x_1c_{j_i}+x_2\gamma(c_{j_i}))}
+ O(2^{-2r}\delta)+O(2^{-r}\delta^2|x|).
\end{multline*}
It follows that there is a constant $C>0$ (independent of $\delta$) such that for each $i$ in the sum, 
\begin{align*}
|T_{\lambda}^{\delta}\psi_{\delta}(x)|\ge 2^{-r-10}\delta
\end{align*}
whenever
\begin{align*}
|x\cdot(c_{j_i}, \gamma(c_{j_i}))|\le C\delta^{-1},\, |x\cdot (1, \gamma^{\prime}(c_{j_i}))|\le C2^{r}.
\end{align*}
It follows that
\begin{align}\label{2ndl1}
\Norm{T_{\lambda}^{\delta}\psi_{\delta}}_1\gtrsim Q^{\prime}=\log(\delta^{-1})^{-1}N(\Omega, \delta)
\end{align}
for $\delta>0$ sufficiently small. Interpolating (\ref{2ndl2}) and (\ref{2ndl1}) gives that 
\begin{align}\label{2ndlp}
\Norm{T_{\lambda}^{\delta}\psi_{\delta}}_p\gtrsim (\log(\delta^{-1}))^{\beta}N(\Omega, \delta)^{\frac{1}{p}}(2^{-r}\delta)^{\frac{p-1}{p}},\qquad 1\le p\le 2
\end{align}
for some $\beta\in\mathbb{R}$. Together (\ref{1stlp}) and (\ref{2ndlp}) imply that
\begin{align}\label{finallp}
\Norm{T_{\lambda}^{\delta}}_{L^p\to L^p}\gtrsim (\log(\delta^{-1}))^{\beta^{\prime}}N(\Omega, \delta)^{\frac{1}{2}}\delta^{\frac{p-1}{p}}
\end{align}
for some $\beta^{\prime}\in\mathbb{R}$. By (\ref{tdelta}), it follows that $\lambda\ge 1-\frac{\kappa_{\Omega}}{2}-\frac{1}{p}$.

\end{proof}

\section{Concluding remarks}\label{concsec}
There are many further questions that arise naturally from the results of this paper; we now discuss a few of them. As previously mentioned, Theorem \ref{2ndmainprop} demonstrates that how ``curved" the boundary of a convex planar domain is, as measured by the parameter $\kappa_{\Omega}$, does not alone determine the $L^p$ mapping properties of the associated Bochner-Riesz operators, but rather there must be other properties of $\Omega$ that play a role. We have seen that domains that satisfy $\mathcal{E}_n(\partial\Omega)=0$ for some $n>2$ can be shown to satisfy $L^p$ mapping properties better than those proved in \cite{sz}. It would be very interesting to construct domains for which $\mathcal{E}_n(\partial\Omega)=0$ for some $n>2$ as well as having an associated set of directions which is $q$-sparse for $q=\frac{n}{n-1}$; for such domains Theorem \ref{genthm} would imply that $p_{\text{crit}}<4/3$. As a simpler preliminary question, it would be already very interesting to construct non-lacunary sets of directions that are $q$-sparse for some $q<2$.
\iffalse{
A related problem that is of interest in its own right is to determine the range of $p$ for which a given directional maximal operator is almost bounded on $L^p$. It would also be interesting to understand the relationship between $L^p$ mapping properties of generalized Bochner-Riesz multipliers and almost boundedness of the corresponding directional maximal operators. We remark that there can be no direct analog of the C\'{o}rdoba-Fefferman theorem (see \cite{cf}) for general convex domains, i.e. it is false that for general convex domains $\Omega$ and $\Theta$ a set of directions associated to $\Omega$ that for $4\le p<\infty$ we have that $M_{\Theta}$ almost bounded on $L^{(p/2)^{\prime}}$ implies that $T_{\lambda}$ is bounded on $L^p$ for all $\lambda>0$. This can be seen by comparing Corollary \ref{cantorcor}  and Theorem \ref{khin1}. It would be interesting to determine whether one could prove a C\'{o}rdoba-Fefferman type theorem if one makes additional assumptions on $\Omega$.}\fi
\newline
\indent
Another question one might also is if for \textit{any} $\kappa\in (0, 1/2)$ (not just for $\kappa$ sufficiently small) we can construct domains for which $p_{\text{crit}}<4/3$. At the very least, we believe that the upper bound on $\kappa_{\Omega}$ in Theorem \ref{2ndmainprop} could be significantly improved with more sophisticated algebraic disjointness constructions than the one used in the proof of Lemma \ref{disjlemma}. In particular, the domains constructed to prove Theorem \ref{2ndmainprop} only exploited algebraic disjointness in one dimension, and it is quite likely that a two-dimensional approach will yield much better results. Finally, it would be interesting if one could determine whether one may prove improved $L^p$ bounds for other certain specific examples of convex domains, such as those with associated directions lying in a standard Cantor set.


\begin{thebibliography}{10}

\bibitem{bateman} M. Bateman, \textit{Kakeya sets and directional maximal operators in the plane}, Duke Math. J. 147 (2009), no. 1, 55-77.

\bibitem{batemankatz} M. Bateman and N. Katz, \textit{Kakeya sets in Cantor directions}, Math. Res. Lett. 15 (2008), no. 1, 73-81.

\bibitem{bnw} J. Bruna, A. Nagel and S. Wainger, \textit{Convex hypersurfaces and Fourier transforms}, Ann. of Math. 127 (1988), 333-365.

\bibitem{car} A. Carbery, \textit{Covering lemmas revisited}, Proc. Edin. Math. Soc. (Series 2) 31 (1988), no. 1, 145-150.

\bibitem{cor} A. C\'{o}rdoba, \textit{A note on Bochner-Riesz operators}, Duke Math. J. 46 (1979), no. 3, 505-511.

\bibitem{cor2} A. C\'{o}rdoba, \textit{The Kakeya maximal function and the spherical summation multipliers}, Amer. J. Math. 99 (1977), no. 1, 1-22.

\bibitem{cor3} A. C\'{o}rdoba, \textit{The multiplier problem for the polygon}, Ann. of Math. 105 (1977) no. 3, 581-588.

\bibitem{cf} A. C\'{o}rdoba and R. Fefferman. \textit{On the equivalence between the boundedness of certain classes of maximal and multiplier operators in Fourier analysis}, Proc. Nat. Acad. Sci. U.S.A. 74 (1977) no. 2, 423.

\bibitem{cgt} A. Carbery, G. Gasper, and W. Trebels, \textit{Radial Fourier multipliers of $L^p(\mathbb{R}^2)$}, Proc. Nat. Acad. 
Sci. U.S.A. 81 (1984), no. 10, Phys. Sci., 3254-3255.

\bibitem{feff} C. Fefferman, \textit{A note on spherical summation multipliers}, Israel J. Math. 15 (1973), 44-52.

\bibitem{feff2} C. Fefferman, \textit{The multiplier problem for the ball}, Annals of Math. (1971), 330-336.

\bibitem{hs} P. Hagelstein and A. Stokolos. \textit{Tauberian conditions for geometric maximal operators.} Trans. Amer. Math. Soc. 361 (2009) no. 6, 3031-3040.

\bibitem{mus} C. Muscalu and W. Schlag, \textit{Classical and Multilinear Harmonic Analysis}, Vols. I-II, Cambridge University Press, 2013.

\bibitem{pod1} A. N. Podkorytov, \textit{Fej\'{e}r means in the two-dimensional case}, Vestnik Leningrad Univ. (Matem.) (1978), no. 13, 32-39, 155 (Russian); English translation in Vestnik Leningrad Univ. Math. 11 (1981).

\bibitem{pod2} A. N. Podkorytov, \textit{Summation of multiple Fourier series over polyhedra}, Vestnik Leningrad Univ. (Matem.) (1980), no. 1, pages 51-58, 119 (Russian); English translation in Vestnik Leningrad Univ. Math. 13 (1983), 69-77.

\bibitem{pod3} A. N. Podkorytov, \textit{Intermediate rates of growth of Lebesgue constants in the two-dimensional case}, in: Numerical Methods and questions in the organizations of calculations, 7, Zapiski Nauchnykh Seminarov Leningradskogo Otdeleniya Matematicheskogo Instituta im. V. A. Steklova AN SSSR 139 (1984), 148-155 (Russian); English translation in Jour. Sov. Math 36 (1987), 276-282.

\bibitem{rubio} J. L. Rubio de Francia, \textit{A Littlewood-Paley inequality for arbitrary intervals}, Rev. Mat. Iberoamericana 1 (1985), 1-14.

\bibitem{sz} A. Seeger and S. Ziesler, \textit{Riesz means associated with convex domains in the plane}, Math. Zietschrifte 236 (2001), no. 4, 643-676.

\bibitem{sjolin} P. Sj\"{o}lin, \textit{Fourier multipliers and estimates of the Fourier transform of measures carried by smooth curves in $\mathbb{R}^2$}, Studia Math. 51 (1974), 170-182.

\bibitem{ss} P. Sj\"{o}gren and P. Sj\"{o}lin, \textit{Littlewood-Paley decompositions and Fourier multipliers with singularities on certain sets}, Annales de l'institut Fourier 31 (1981) no. 1, 157-175.

\bibitem{stein} E. M. Stein, \textit{Harmonic analysis: Real variable methods, orthogonality and oscillatory integrals}, Princeton Univ. Press, 1993.

\bibitem{taovu} T. Tao and V. H. Vu, \textit{Additive combinatorics}, Vol. 105, Cambridge Univ. Press, 2006.

\bibitem{trebels} W. Trebels, \textit{Some Fourier multiplier criteria and the spherical Bochner-Riesz kernel}, Rev. Roumaine Math. Pures. Appl. 20 (1975), 1173-1185.

\bibitem{zyg} A. Zygmund, \textit{Trigonometric series}, Cambridge University Press, Paperback edition 1988.






\end{thebibliography}
\end{document}